\documentclass[leqno,a4paper]{article}
\usepackage{amsmath,amsthm,amssymb}

\newcommand{\R}{{\mathbb R}}

\newcommand{\K}{{\mathbb K}}

\newcommand{\Ga}{{\Gamma}}
\newcommand{\De}{{\Delta}}
\newcommand{\na}{{\nabla}}

\newcommand{\ou}{{\underline{u}}}

\newcommand{\n}{{\underline{n}}}

\newcommand{\ta}{{\underline{t}}}
\newcommand{\oo}{{\underline{\omega}}}

\newcommand{\pa}{{\partial}}
\newcommand{\m}{{\mu}}

\newcommand{\cq}{{\widehat{q}}}

\newcommand{\cC}{{\widehat{C}}}

\newtheorem{definition}{Definition}[section]

\newtheorem{theorem}{Theorem}[section]
\newtheorem{lemma}[theorem]{Lemma}

\numberwithin{equation}{section}
\def\vs1{\vspace{1ex}}

\def\O{\Omega}
\def\pa{\partial}
\def\dy{\displaystyle}

\def\be{\begin{equation}}
\def\ba{\begin{array}}
\def\ea{\end{array}}
\def\ee{\end{equation}}

\begin{document}
\title{\bf\ On the singular $p$-Laplacian system under Navier slip type boundary conditions. The
gradient-symmetric case.}
\author{ H.~Beir\~ao da Veiga \thanks{Dipartimento di Matematica,
 Universit\`{a} di Pisa, Via Buonarroti 1/C, 56127 Pisa, Italy,
 email: bveiga@dma.unipi.it}
 }
%\date{}
\maketitle
\begin{abstract} We consider the $\,p$-Laplacian system of $N$ equations in $n$ space variables, $\,1<\,p\leq\,2\,,$
under the homogeneous Navier slip boundary condition. Furthermore,
the gradient of the velocity is replaced by the, more physical,
symmetric gradient. We prove $\,W^{2,\,q}\,$ regularity, up to the
boundary, under suitable assumptions on the couple $\,p,\,\,q\,$.
The singular case $\mu=\,0$ is covered.
\end{abstract}

\vspace{0.2cm}


%\medskip

%\noindent \textbf{Mathematics~Subject~Classification~(2000):} 35B65,35J57.

\section{Introduction. The main result}
We consider the system
\begin{equation}\label{NSC}
-\,\nabla \cdot
\,\big(\,(\,\m+|\,D\,u|^2\,)^{\frac{p-2}{2}}\,D\,u\,\big) =\,f \
\mbox{ in } \O
\end{equation}
under the Navier slip boundary condition \eqref{bcns}. Here, and in
the sequel, $u$ is an $N$-dimensional vector field defined in a
bounded, open, connected, subset  $\,\O\,$ of $\R^n\,,$ locally
situated on one side of its boundary, a smooth manifold $\Ga$. We
denote by $\n$ the outer unit normal to $\partial\Omega\,$ and by
$\mu\,\geq \,0$ a given parameter. The vector field $f$ is given.
For convenience, we will assume that $\O\,$ has not axis of
symmetry. The reason will be clear below. We are particularly
interested in the singular case
$\m=\,0\,.$\par%
By
$$
D \,u=\,\nabla u+\,\nabla^{T} u
$$
we denote the symmetric gradient. So
\begin{equation}
D_{i\,j}(u)=\,\pa_i\,u_j +\,\pa_j\,u_i\,,
\end{equation}
 where $i,\,j=\,1,2,...,n\,$. Often we simply write $D$, provided that
the vector field under consideration follows from the context.\par%
Equation \eqref{NSC} has been considered by mathematicians mostly
with $\,D \,u\,$ replaced by $\,\nabla u\,.$ It is worth noting that
\eqref{NSC} satisfies the Stokes Principle (see \cite{stokes}, and
\cite{serrin} page 231), a significant physical requirement of
isotropy,
which does not hold if we replace $\,D \,u\,$ by $\,\nabla u\,.$\par%
In the following we denote by  $\ta(u)\,$ the Cauchy stress vector
$$
\ta(u)=\,(\,D\,u)\cdot\,\n\,.
$$
So,
$$
t_j=\,(\pa_i\,u_j +\,\pa_j\,u_i\,)\,n_i\,,
$$
where (here and in the sequel) we use the summation convention on repeated indexes.\par%
The homogeneous Navier slip type boundary condition, see
\cite{Navier}, says that the velocity is tangent to the boundary,
and the tangential component of the stress vector $\ta(u)\,$
vanishes on the boundary. We write this condition in the following
form
\begin{equation}\label{bcns}
\left\{
\begin{array}{l}
\ou\cdot \n=\,0,\\
(\ta(u))_\tau=\,0\,,
\end{array}
\right.
\end{equation}
where in general the subscript $\tau$ denotes tangential component.
For a mathematical study of the above boundary condition see, for
instance, \cite{bvadvances} and \cite{so-sca}, where this boundary
condition is associated to the linear Stokes problem.\par%
By $L^p(\O)$ and $W^{m,p}(\O)$, $\,1 \leq\,p \leq\,\infty\,$,
$\,m\,$ nonnegative integer, we denote the usual Lebesgue and
Sobolev spaces, with the standard norms $\|\cdot\|_p\,$ and
$\|\,\cdot\,\|_{m,p}\,$, respectively.\par%
In notation concerning  norms and functional spaces, we do not
distinguish between scalar and vector fields. For instance
$L^p(\Omega;\R^N)= [L^p(\O)]^N$, $N>1$, is denoted simply by
$L^p(\O)$. We define
$$
V_p=\,V_p(\O)=\,\{ v \in\,W^{1,\,p}(\O):\, v\cdot \n=\,0
\,\textrm{\,on\,}\, \Ga\,\}.
$$
The linear space $\,V_p(\O)\,$, endowed with one of the following
norms
$$
\big(\,\|\,v\,\|_p +\,\|\,D\,v\,\|_p\,\big)^\frac1p \,,
\quad\big(\,\|\,v\,\|_p +\,\|\,\na\,v\,\|_p\,\big)^\frac1p \,,\quad
\|\,\na\,v\,\|_p\,,
$$
is a Banach space. The above norms are equivalent in $\,V_p(\O)\,.$
Further, since we assume that the domain $\,\O\,$ has not axis of
symmetry, it follows that $\,\|\,D\,v\,\|_p\,$ alone is a norm. For
a quite complete discussion on this point, we refer to
\cite{bvadvances}. Without this hypothesis, existence or uniqueness
of the solution may fail, depending on the particular external force
$\,f\,$. We believe that this should be not difficult to show, by
appealing to counter examples. However, a complete study of the
possible phenomena (due to nonlinearity) should be difficult but
quite interesting. Note that, if we replace  $\,D \,u\,$ by
$\,\nabla u\,,$ the above assumption on $\,\O\,$
is superfluous.\par For the proof of Korn's inequality we refer,
for instance, to \cite{so-sca} and \cite{pares}.\par%
Our main result is the following.
\begin{theorem}\label{teoremaq}
Assume that $\mu\geq 0\,,$ and let $f\in L^q(\O)\,,$ where
$\,q>\,n\,.$ Let $C_q=\,C(q,\,\O)\,$ be the constant that appears in
the linear estimate \eqref{tse} below. Assume that
\begin{equation}\label{cqom}%
(2-\,p)\,C_q <\,1 \,.
\end{equation}
Then, the weak solution $\,u\,$ to the problem \eqref{NSC},
\eqref{bcns} belongs to $\,W^{2,q}(\O)\,$. Moreover, the following
estimate holds
\begin{equation}\label{dnq}
 \|u\|_{2,q}\leq C
 \,\left(\|f\|_q+\|f\|_{q}^\frac{1}{p-1}\right)\,.
\end{equation}
\end{theorem}
The proofs also apply, in a simpler way, to the Dirichlet boundary
value problem. In section \ref{ssei} we consider the boundary value
problem \eqref{bcnos}. The singular case remains open. Finally, we
refer to section \ref{ssetes} for an application to the
fluid mechanics system \eqref{NSCCC}. %

\vspace{0.2cm}

Regularity of solutions for systems like \eqref{NSC} has received
substantial attention from many authors. We refer, for instance, to
references \cite{acerbi}, \cite{Hamburger}, \cite{Lieb88},
\cite{Liu}, \cite{Tolk2}, \cite{ura}. Other related results may be
found in \cite{anton-sh-1}, \cite{anton-sh-2}, \cite{BDVCRIplap},
\cite{DB}, \cite{DBM}, \cite{FS}, \cite{MP}, and references therein.\par%
\vspace{0.2cm}

%The very basic structure of this paper follows that introduced, in a
%more difficult context, in reference \cite{bvlali}.%
The plan of the paper is the following: In section \ref{sdue} we
recall the existence and uniqueness result of the weak solution. In
section \ref{stre} we introduce an auxiliary linear problem and
state (by appealing to well know classical results) the existence of
solutions to this linear problem in spaces $W^{2,\,q}(\O)\,.$ In
section \ref{stre} we formulate the non-linear problem in a more
explicit, formally equivalent, form in which the non-linearities are
(roughly speaking) concentrated in the right hand side (see equation
\eqref{formfina} below). Furthermore, we appeal to this formulation
to define "strong solution". In section \ref{squattro}, by assuming
$\,\m>\,0\,$ and by appealing to the result stated in section
\ref{stre} for the auxiliary linear problem, we show that the strong
solution introduced in section \ref{squattro} exists and belongs to
$\,W^{2,\,q}(\O)\,,$ for each $\,\m>\,0\,.$ Moreover, the estimates
obtained are independent of $\,\m\,.$ This last property allows us
to extend, in section \ref{scinque}, the regularity result to the
singular case $\,\m=\,0\,$ by passing to the limit in the
variational formulation \eqref{bufi} as $\,\m\,$ tends to zero. In
section \ref{ssei} we consider the boundary value problem
\eqref{bcnos}. Finally, in  section \ref{ssetes}, we appeal to a
recent result proved by Petr Kaplick\'y and Jakub Tich\'y, to show
that the result claimed in theorem \ref{teoremaq} applies to
solutions to the system \eqref{NSCCC}, in the particular case
$\,q=\,\cq\,$, see \eqref{errq2}.
\section{Existence and uniqueness  of the weak solution}\label{sdue}
Existence and uniqueness of weak solutions follows from well know
results. Let us recall some basic points. Set
\begin{equation}\label{buh}
B(\,D\,u\,)=\,(\,\m+|\,D\,u|^2\,)^{\frac{p-2}{2}}\,.
\end{equation}
By appealing to the identity  $\,D_{ij}u \,D_{ij}v=\,2\,D_{ij}u
\,\pa_j v_i\,,$ integration by parts shows that
\begin{equation}\label{bufe} \ba{ll}\vspace{1ex}\displaystyle%
\frac12 \,\int_{\O}\ \ B (D\,u)\, \cdot D\,u  \cdot D\, v \,dx
=-\,\int_{\O}\,\na \cdot\,\big(\,B(D\,u)\,D(u)\,\big) \cdot\,v
\,dx\,  \\
\dy +\,\int_{\Ga} \,B(D\,u)\,[\,(D\,u)\,\cdot\,v \,\cdot\,n\,]
\,dS\,.%
\ea%
\end{equation}
Hence,
\be\label{bufi} \ba{ll}\vspace{1ex}\displaystyle%
\frac12 \,\int_{\O}\ \ B (D\,u)\, \cdot D\,u \cdot D\, v \,dx
=\,-\,\int_{\O}\,\na \cdot\,\big(\,B(D\,u)\,D(u)\,\big)
\cdot\,v \,dx\,  \\
\dy +\,\int_{\Ga} \,B(D\,u)\,(\ta(u))_{\tau}\cdot v \, dS\,,
\ea\ee%
provided that $\,v\cdot\,n=\,0\,$ on $\,\Ga\,.$\par%
This last identity justifies the following definition.
\begin{definition}\label{noit}
Let $f \in V'_p(\O)$. We say that $u$ is a {\rm weak solution} of
problem \eqref{NSC}, \eqref{bcns} if $\,u\,\in\, V_p(\O)$ satisfies
\be\label{bufi} \ba{ll}\vspace{1ex}\displaystyle%
\frac12 \,\int_{\O}\ \ B (D\,u)\, \cdot D\,u \cdot D\, v \,dx
=\,\int_{\O}\, f\, \cdot\, v \,dx\,,
\ea\ee%
for all $\,v \in \,V_p(\O)\,$.
\end{definition}
Existence and uniqueness of the above weak solution, for any fixed
$\,\m \geq\,0\,,$ follows by appealing to the theory of monotone
operators, see J.-L. Lions \cite{lions}.%
\section{An auxiliary linear problem}\label{stre}
In this section we consider the linear problem
\begin{equation}\label{NSC2}
-\,\nabla \cdot \,\big(\,D\,u\,\big) =\,F \ \mbox{ in } \O
\end{equation}
under the boundary condition \eqref{bcns}, and state an auxiliary
result to be used in the next sections. This particular result
follows from well known general results.\par%
Note that equation \eqref{NSC2} may be also written in the
equivalent form (not used in this section)
\begin{equation}\label{forteF}%
-\,\Delta\,u -\, \nabla\,(\nabla \cdot\,u\,)= F\,.
\end{equation}%
\begin{definition}\label{noit}
Let $f \in V'_2(\O)$. We say that $u$ is a {\rm{weak solution}} of
problem \eqref{NSC2}, \eqref{bcns} if $\,u\,\in\, V_2(\O)$ satisfies
\be\label{basta}%
\frac12 \,\int_{\O}\ \  D\,u \cdot D\, v \,dx =\,\int_{\O}\, F\,
\cdot\, v \,dx\,,
\ee%
for all $\,v \in \,V_2(\O)\,$.
\end{definition}
Coerciveness of the bilinear form on the left hand side of
\eqref{basta} follows here by appealing to the fact that
$\,\|\,D\,v\,\|\,$ alone is a norm in $\,V_2(\O)\,$, since we have
assumed that $\O$ has not axis of symmetry. Hence, existence,
uniqueness, and the standard estimate holds for the above problem.

\vspace{0.2cm}

Next we consider the regularity of the solutions to the above linear
system \eqref{NSC2} (or, equivalently, \eqref{forteF}) under the
boundary condition \eqref{bcns}. The $\,W^{2,\,2}(\O)\,$ regularity
may be proved, for instance, by following \cite{so-sca} and
\cite{bvadvances}. The reader may easily adapt the argument
developed in \cite{so-sca}, section 4. Further, as claimed in
reference \cite{so-sca} section 4, by appealing to results proved in
reference \cite{so2} (see also \cite{a-d-n}), the
$\,W^{2,\,q}(\O)\,$ regularity, for arbitrarily large exponents $q$,
follows. Actually, under suitable, canonical, regularity assumptions
on $F$ and $\O\,,$ $\,W^{m,\,q}(\O)\,$ regularity for arbitrarily
large values of $m\geq\,2\,$ follows.\par%
Alternatively, we may follow \cite{Sol71} to show that the system
\eqref{forteF}, \eqref{bcns} is of Petrovks\u\i\ type (a subclass of
Agmon-Douglis-Nirenberg elliptic systems). See, in particular, the
Theorem 5.1 in reference \cite{Sol71}. This allows a simplified
integral representation formula for the solutions to the above
linear problem. Moreover, for Petrovks\u\i's systems, the
$W^{2,\,2}$-regularity yields full $W^{m,\,q}$-regularity, provided
that the data are sufficiently smooth. In particular, there is a
constant $\widetilde{C}_q \,$ such that
\begin{equation}\label{tse2}%
\|\,u\,\|_{2,\,q} \leq\,\widetilde{C}_q \,\|\,F\,\|_q \,.
\end{equation}
Summarizing, the following result holds.
\begin{theorem}\label{seva}
Consider the linear boundary value problem \eqref{NSC2},
\eqref{bcns}. Assume that $\,F \in\,L^q(\O)\,,$ for some $\,q
\,\geq\,2\,.$ Then, the solution $u$ to the above linear problem
belongs to $W^{2,\,q}(\O)\,.$ Furthermore, there is a constant
$\,C_q =\,C_q(q,\,\O)\,,$ such that
\begin{equation}\label{tse}%
\|\,\na\,Du\,\|_q \leq\,C_q\,\|\,F\,\|_q \,.
\end{equation}
\end{theorem}
Clearly, $\,C_q \leq\,\widetilde{C}_q\,.$ The pointwise estimate
\begin{equation}\label{dirup}
|\nabla^2\,u| \leq\,3\,|\nabla\, D\,u| \leq\,6\,|\nabla^2\,u|
\end{equation}
shows that $\|\,u\,\|_{2,\,q}$ and $\,\|\,\na\,Du\,\|_q \,$ are
equivalent norms in $\,W^{2,\,q}(\O) \cap\,V_q(\O)\,.$

\vspace{0.2cm}

Under the homogeneous Dirichlet boundary value problem the constant
$\,C_q\,$ is bounded from above by a constant $\,K\,$ times $\,q\,.$
Clearly, this nice behavior can not hold, with the same constant
$\,K\,,$ for arbitrarily large values $q$. To each upper-bounded
interval of values $\,q\,$ it corresponds a distinct value $\,K\,.$
See \cite{yud}. We do not know whether a similar (quite predictable)
result is known for the Navier boundary condition.
\section{The strong solution. Definition.}
The main lines followed in this section have their starting point in
some ideas already used, in a more complex context, in reference
\cite{bvlali} (see, for instance, equations (4.17), (4.25), (4.26),
and (4.27) in this last reference). Since
$$
\nabla\,(\,\m+|\,D\,u|^2\,)^{\frac{p-2}{2}}=\,\frac{p-2}{2}\,
(\,\m+|\,D\,u|^2\,)^{\frac{p-4}{2}}\,\nabla\,(\,|D\,u|^2)\,,
$$
straightforward calculations show that%
\be\label{formula}\ba{ll}\vspace{1ex}\displaystyle \nabla \cdot
\,\big(\,(\,\m+|\,D\,u|^2\,)^{\frac{p-2}{2}}\,D\,u\,\big)=\,
(\,\m+|\,D\,u|^2\,)^{\frac{p-2}{2}}\,
\,\nabla \cdot (Du\,) \\
\dy +\,(p-2)\,(\,\m+|\,D\,u|^2\,)^{\frac{p-4}{2}}\,I(u)%
\ea\ee%
where, by definition,
$$
I(u)=\,\frac12 \,\nabla\,(\,|D\,u|^2)\cdot\, D\,u=\,(\,Du :\,\nabla
Du\,)\cdot\,Du\,.
$$
The $j$ component of the vector field $I(u)$ is given by
$$
I_j(u)=\,\sum_k \,\sum_{l,\,m}\, D_{l m} (\pa_k\,D_{l m}) \,D_{k
j}\,.
$$
By improving an argument already used in \cite{bvlali}, we may prove
(as in the proof of Lemma 3.4 in \cite{nonhom}) the algebraic
relation
$$
|I\cdot \,\xi| \leq\,|D|^2\,|\nabla \,D\,u|\,|\xi|\,,
$$
for each arbitrary vector field in $\,\xi\in\,\R^N\,$, where
$$
|\nabla \,D\,u|^2=\,\sum_{m,l,k} \,(\pa_k\,D_{m\,l})^2\,.
$$
Consequently, the pointwise estimate
\begin{equation}\label{buf2}
|I(u)| \leq\,|D|^2\,|\nabla \,D\,u|
\end{equation}
holds.\par%
Next we introduce the notion of strong solution used in the next
section.
\begin{definition}\label{noitidia}
Assume that $\,\m >\,0\,,$ and let $\,f \,\in \, L^q(\O)\,$ be
given, $q>1\,$. We say that $\,u\in\,W^{2,\,q}(\O)\,$ is a
{\rm{strong solution}} of problem \eqref{NSC}, \eqref{bcns} if
$\,u\,$ satisfies \eqref{bcns} in the trace sense and, moreover,
 the equation
\be\label{formfina}%
-\,\nabla \cdot\,(D\,u) =(p-\,2)\,G(u)+\,
(\,\m+|\,D\,u|^2\,)^{\frac{p-2}{2}}\,f%
\ee%
holds almost everywhere in $\O\,,$ where
$$
G(v)=\,(\,\m+|\,D\,v|^2\,)^{-1}\,I(v)\,.
$$
\end{definition}
Note that $ G(v) \leq\,|\nabla \,D\,v|\,,$ almost everywhere in
$\,\O\,,$ for all $\m\,$. So,
\begin{equation}\label{geve}
\|\,G(v)\,\|_q \leq\,\|\nabla \,D\,v\|q\,.
\end{equation}
\section{Existence of the strong solution for $\m>\,0\,.$}\label{squattro}
Fix $\,\m>\,0\,,$ and let $\,f \in \,L^q(\O)\,$. Following
\cite{BVCRI}, by appealing to a fixed point argument, one proves the
existence of a (unique) strong solution $\,u\in\,W^{2,\,q}(\O)\,$ of
the above problem. Let us sketch the proof.

\vspace{0.2cm}

Since $q>n\,,$ there is a constant $\,\cC(q,\,\O)\,$ such that
\be\label{c7}%
\|\,D\,v\|_{\infty}\leq \,\cC\,\|\nabla\,D\,v\,\|_q\,,%
\ee%
for all $v\,\in\,W^{2,\,q}(\O) \cap\, V_q(\O)\,.$  Hence,
\begin{equation}\label{holdes}%
\|\,|\,D v |^{2-p}\,f\|_q\leq \dy\|\,D v\,\|_{\infty}^{2-p}\,
\|\,f\|_q \,
\leq\,\cC^{2-\,p}\,\|\nabla\,D\,v\,\|^{2-\,p}_q\,\|\,f\|_q\,.
\end{equation}%
Further, since $\,(a+\,b)^\alpha \leq \,a^\alpha + b^\alpha\,$
for nonnegative $a$ and $b$, and  $ 0 <\alpha<1\,,$ it follows that%
\begin{equation}\label{cjen}
(\m+\,|D\,v|^2)^{\frac{2-p}{2}}\leq
\,\mu^{\frac{2-p}{2}}+\,|D\,v|^{2-p}\,.%
\end{equation}%
From \eqref{holdes} and \eqref{cjen} we show that
\be\label{lap1w}%
\|\,(\m+\,|\,D\,v\,|^2)^{\frac{2-p}{2}} \,f\|_q \leq
\,\mu^{\frac{2-p}{2}}\,\|\,f\,\|_q
+\,\cC^{2-\,p}\,\,\|\,\na\,D v\,\|^{2-\,p}_q\,\|\,f\|_q\,.%
\ee%
Next we define the convex closed set
\begin{equation}\label{krapa}%
\K=\,\K(R)=\,\{ v\in\,W^{2,\,q}(\O)\cap\,V_q(\O)\,:\,\| \na D\,v\|_q
\leq\,R\,\}\,,
\end{equation}%
and consider, for each $v \in\,\K\,,$ the solution $u=\,T(v)\,$ to
the problem
\be\label{formfina3}%
-\,\nabla\,\cdot\,D\,u\,=\,F(v)\equiv (p-\,2)\,G(v)+\,
\,(\,\m+|\,D\,v|^2\,)^{\frac{2-\,p}{2}}\,f\,,%
\ee%
under the boundary conditions \eqref{bcns}\,.\par%
By appealing to equations \eqref{tse}, \eqref{geve}, and
\eqref{lap1w}, we obtain the estimate
\begin{equation}\label{bingas}%
\|\,\na\,Du\,\|_q \leq\,C_q\,\{\,(2-\,p)\,\|\,\na\,Dv\,\|_q +\,
\mu^{\frac{2-p}{2}}\,\|\,f\,\|_q
+\,\cC^{2-\,p}\,\,\|\,\na\,D v\,\|^{2-\,p}_q\,\|\,f\|_q\,\}\,.%
\end{equation}
Next we show that if $\|\,\na\,Dv\|_q\leq R\,$ then the
corresponding solution $u=\,T(v)$ satisfies the same estimate,
namely $\|\,\na\,Du\|_q\leq R\,$. This shows that $T(\K) \subset \,\K\,.$\par%
Since $\,v \in \K$ it follows that%
\be\label{ssim}%
\|\,\na\,Du\,\|_q \leq \,\mu^{2-p}\,C_q\,\|f\|_q\,+\,(2-p)\,\,C_q\,R
\,+\,C_q\, \cC^{2-\,p} \,\|\,f\|_q \,R^{2-\,p}\,.
\ee%
By assuming \eqref{cqom}, we show that $\,u \in\,\K(R)\,$ if
$$
[\,1-\,(2-\,p)\,C_q\,]\,R \geq \,
\,\mu^{2-p}\,C_q\,\|f\|_q\,+\,C_q\, \cC^{2-\,p} \,\|\,f\|_q
\,R^{2-\,p}\,.
$$
This inequality is satisfied if, for instance, its left hand side is
equal to two times the sum of the two terms on the right hand side.
This holds for
\begin{equation}\label{condidois}
R=\,\frac{2}{\alpha}\,\mu^{2-p}\,C_q\,\|f\|_q +\,
(\,\frac{2\,C_q\,\cC^{2-\,p}}{\alpha}\,)^{\frac{1}{p-\,1}}\,\|\,f\|^{\frac{1}{p-\,1}}_q\,,
\end{equation}
where $\,\alpha=\,1-\,C_2(q)\,(2-\,p)\,.$ Hence $\|\,\na\,Du\,\|_q
\leq\,R\,,$ and the inclusion $\,T(\K) \subset \,\K\,$ follows. This
is the main ingredient to prove the existence of a fixed point in
$\,\K\,.$ For the missing details we refer to
the argument developed in reference \cite{BVCRI}.\par%
The expression of $\,R\,$ shows that the uniform estimate
\eqref{dnq} follows. Actually, we have shown that,
for each positive $\m\,,$ the estimate%
\begin{equation}\label{dnq2}
\|u^\m\|_{2,q}\leq C \,\left(\|f\|_q+\|f\|_{q}^\frac{1}{p-1}\right)
\end{equation}
holds, where $\,u^\m\,$ denotes the strong solution related to the
particular positive value $\,\m\,$.
\section{Existence of the strong solution for $\m=\,0\,.$}\label{scinque}
In this section, since the estimate \eqref{dnq2} is uniform with
respect to values $\,\m\,$ (assumed to be bounded from above), by
appealing to a compactness argument, we pass to the limit, as $\m$
tends to zero, in the weak formulation \eqref{bufi} (which contains
the singular case $\m=0$) and prove that the weak solution $u$ to
the singular problem also belongs to $W^{2,q}(\O)\,,$ and satisfies
\eqref{dnq}.\par%
We start by recalling the definition of weak solution $\,u^\m\,$ of
problem \eqref{NSC}, for $\,\m \geq\,\,0\,:$
\begin{equation}\label{buf23}
\int_{\O}\
\left(\,\m+|\,D\,u^\m|^2\,\right)^{\frac{p-2}{2}}\,D\,u^\m\, \cdot
D\, v \,dx\ =\,\int_{\O}\, f\, \cdot\, v \,dx\,,
\end{equation}
for all $\,v \in \,V_p(\O)\,.$ This condition is satisfied by the
strong solutions $\,u^\m\,,$ for $\,\m >\,0\,,$ constructed in the
previous section. Since these solutions are uniformly bounded in
$W^{2,\,q}(\O)\,,$ suitable sub-sequences, which we continue to
denote by $\,u^\mu\,$, weakly converge in $W^{2,\,q}(\O)\,$ to some
$\,u\,$. The argument followed in \cite{BVCRI} shows
that we may pass to the limit in \eqref{buf23} to prove that%
\begin{equation}\label{buf24} \int_{\O}\
\,|\,D\,u|^{\frac{p-2}{2}}\,D\,u\, \cdot D\, v
\,dx\ =\,\int_{\O}\, f\, \cdot\,v \,dx\,,%
\end{equation}
for all $\,v \in \,V_p(\O)\,.$ So, $\,u \in\,W^{2,\,q}(\O)\,$ is the
solution (known to be unique), corresponding to $\m=\,0\,.$ To prove
the above claim, we have to show that, for each fixed $\,v \in
\,V_p(\O)$, the left hand side of equation \eqref{buf23} converges
to the left hand side of \eqref{buf24}. Essentially, the proof
followed in reference \cite{BVCRI} section 4 applies here. For the
reader's convenience, we repeat the main argument here.\par%
Since $\,u^\mu\,\rightharpoonup u\,$ weakly in $W^{2,q}(\O)\,$, and
$\,q>\,n\,,$ strong convergence (of suitable subsequences) in
$\,W^{1,s}(\O)\,$, for any $\,s\,,$ follows. So, strong convergence in $W^{1,p}(\O)$ holds.\par%
We write the integral on the left-hand side of \eqref{buf23} as
\be\label{aaa2} \int_{\Omega} \,\big[\,\left(\,\mu+|\,D
u^\mu|^2\,\right)^{\frac{p-2}{2}}\,D u^\mu\, - \,\left(\,\mu+|\,D
u\,|^2\,\right)^{\frac{p-2}{2}}\,D u\,\big] \cdot D v \,dx \ee
$$
 +\,\int_{\Omega} \,\left(\,\mu+|\,D
u|^2\,\right)^{\frac{p-2}{2}}\,D u\, \cdot \,D v \,dx\,,
$$
and show that the first integral tends to zero, and the second
integral tends to the left hand side of \eqref{buf24}. The
inequality%
\be\label{tensorS1}|\,(\mu+|A|)^{\frac{p-2}{2}}
A-(\mu+|B|)^{\frac{p-2}{2}} B| \leq\,
C\,\frac{|A-B|}{(\mu+|A|+|B|)}{\!\atop ^{{2-p}}}\,,\ee%
where $C$ is independent of $\mu$ (see \cite{DER} equation (6.8)),
shows that the absolute value of the first integral in equation
\eqref{aaa2} is bounded by%
$$
C\,\int_{\Omega}\! \,\left(\,\mu+|\,D\,u\,|+|\,D\,
u^\mu|\,\right)^{p-2}\,|\,D\,u-\,D u^\mu\,|\,|\,D \,v|\, dx\,.
$$
Since
$$
\left(\,\mu+|D u\,|+|D u^\mu|\,\right)^{p-2}\,|\,D u-\,D u^\mu\,|
\leq\,\left|\,D u-\,D u^\mu\,\right|^{p-1}\,,
$$
the absolute value of the first integral in equation \eqref{aaa2}
is bounded by%
$$
C\,\|\,D\,
u^\mu\!-\,D\,u\,\|_p^{p-1}\,\|\,D\,v\,\|_p\,,%
$$
which tends to zero with $\,\mu\,.$\par%
Finally,
$$
 \lim_{\mu\to 0^+}\int_{\Omega}
\,\left(\,\mu+|\,D\,u\,|^2\,\right)^{\frac{p-2}{2}}\,D\,u\, \cdot
\,D\,v  \,dx= \,\int_{\Omega} \,|\,D\, u\,|^{p-2}\,\,D
u \cdot \,D\,v \, dx\,,%
$$
by Lebesgue's dominated convergence theorem.\par%

\section{On a related slip boundary condition}\label{ssei}
In this section we consider the system \eqref{NSC} under the slip
boundary condition
\begin{equation}\label{bcnos}
\left\{
\begin{array}{l}
\ou\cdot \n=\,0,\\
\oo(u) \times \,\n=\,0\,, \; \textrm{on} \; \Ga\,,
\end{array}
\right.
\end{equation}
where $\,\oo=\,\oo(u)=\, curl \,\ou\,,$ and $\,n=\,N=\,3\,.$ The
boundary condition \eqref{bcnos}, introduced by C. Bardos in
reference \cite{bardos}, has been studied by a large number of
authors.\par%
In the following preliminary approach to the above problem, we start
by a sketch of the proof of the existence of the strong solution,
under the assumption $\,\m>\,0\,$. Moreover, an uniform
$\,W^{2,\,q}(\O)\,$ estimate is claimed. However, the case
$\,\m=\,0\,$ is not considered here, due to the lack of a suitable
definition of weak solution. This point will be discussed below.\par%
Concerning the existence of a strong solution for each $\,\m>\,0\,,$
by taking into account definition \ref{noitidia}, and by appealing
to \eqref{NSC} and \eqref{formula}, we say that $u$ is a {\rm{strong
solution}} of problem \eqref{NSC}, \eqref{bcnos} if
$\,u\in\,W^{2,\,2}(\O)\,$ enjoys the boundary condition
\eqref{bcnos}, and satisfies equation \eqref{formfina} almost
everywhere in $\O\,$. We write here the
equation \eqref{formfina} in the equivalent form%
\be\label{formfina2}%
\ba{ll}\vspace{1ex}\displaystyle%
-\,\Delta\,u -\, \nabla\,(\nabla \cdot\,u\,)= \\
\dy (p-\,2)\,(\,\m+|\,D\,u|^2\,)^{-1}\,I(u)+\,
(\,\m+|\,D\,u|^2\,)^{\frac{p-2}{2}}\,f\,.%
\ea\ee%

\vspace{0.3cm}

Let us prove the following result.
\begin{lemma}
The following identity holds.

\be\label{ospois} \ba{ll}\vspace{1ex}\displaystyle%
\int_{\O} \,\nabla\,(\nabla \cdot\,u\,)\cdot \,\De\,v
\,dx=
\\
\dy \,\int_{\O} \,\na\,(\nabla \cdot\,u\,)\cdot \,\na\,(\nabla
\cdot\,v\,) \,dx  -\, \int_{\Ga} \,\nabla\,(\nabla
\cdot\,u\,)\cdot\,(\, \n \times \,\oo(v)\,) \,d\Ga\,.%
\ea\ee%
In particular
\begin{equation}\label{ochey3}
\int_{\O} \,\nabla\,(\nabla \cdot\,u\,)\cdot\,\Delta\,u \,dx=\,
\int_{\O} \,|\nabla\,(\nabla \cdot\,u\,)|^2 \,dx  -\, \int_{\Ga}
\,\nabla\,(\nabla \cdot\,u\,)\cdot\,(\, \n \times \,\oo(u)\,)
\,d\Ga\,.
\end{equation}
\end{lemma}
\begin{proof}
Since
$$
\De \,v =\,\nabla\,(\nabla \cdot\,v\,)-\, \na \times\,\oo(v)\,,
$$
it follows that%
\begin{equation}\label{antes}%
\int_{\O} \,\nabla\,(\nabla \cdot\,u\,)\cdot \,\De\,v
\,dx=\,\int_{\O} \,\na\,(\nabla \cdot\,u\,)\cdot \,\na\,(\nabla
\cdot\,v\,) \,dx -\,\int_{\O}
\,\nabla\,(\nabla \cdot\,u\,)\cdot\, \na \times \,\oo(v)\, dx\,.%
\end{equation}%
On the other hand, by appealing to the identity
$$
\int_{\O} \,f\cdot\,(\na \times\,g\,) \,dx =\,\int_{\O} \,(\na
\times\,f\,) \cdot \,g \,dx +\,\int_{\Ga} \,f\cdot\,(\n \times\,g\,)
\,d\Ga\,,
$$
we get
$$
\int_{\O} \,\nabla\,(\nabla \cdot\,u\,)\cdot\, \na \times \,\oo(v)\,
dx=\, \int_{\Ga} \,\nabla\,(\nabla \cdot\,u\,)\cdot\,(\, \n \times
\,\oo(v)\,) \,d\Ga\,.
$$
\end{proof}
Note that the boundary integral in equation \eqref{ochey3} vanishes
if $\,u\,$ satisfies the boundary condition $\,(\na \times\,u)
\times \,\n=\,0\,$ on $\Ga$. So
\begin{equation}\label{ochey}
\int_{\O} \,\nabla\,(\nabla \cdot\,u\,)\cdot\,\Delta\,u \,dx=\,
\int_{\O} \,|\nabla\,(\nabla \cdot\,u\,)|^2 \,dx\,.
\end{equation}
Multiplication of both sides of \eqref{formfina2} by
$\,\Delta\,u\,,$ followed by integration in $\,\O\,,$ together with
\eqref{ochey} and \eqref{buf2}, leads to the following a priori
estimate, uniform with respect to $\,\m>\,0\,.$

$$
\ba{ll}\vspace{1ex}\displaystyle%
\|\Delta\,u\|^2_2 +\,\|\nabla\,(\nabla \cdot\,u\,)\|^2_2 \leq
\\
\dy (2-\,p)\,\int_{\O} \,|\nabla \,D\,u|\,\cdot\,\Delta\,u \, dx \,+%
\int_{\O} \,(\,\m+|\,D\,u|^2\,)^{\frac{2-\,p}{2}} \,|\,f\,|
\,|\Delta\,u| \,dx\,.%
\ea%
$$
This shows that, for each $\,\m> \,0\,,$ it should be not difficult
to prove the existence of a strong solution $\,u\,$ in
$\,W^{2,\,2}(\O)\,,$ under the smallness assumption on $\,2-\,p\,.$
By appealing to Agmon-Douglis-Nirenberg results, this should lead to
an estimate of $\,u\,$ in $\,W^{2,\,q}(\O)\,,$ uniform with respect
to $\,\m>\,0\,.$ Clearly, we have to assume a restriction (like
\eqref{cqom}) relying the exponents $\,q\,$ and $\,p\,.$ The
estimates obtained are uniform with respect to $\,\m>\,0\,$.
However, a suitable definition of weak solution, for $\,\m
\geq\,0\,,$ must be established, as a previous step to try to
"pass to the limit" as $\,\m\,$ goes to zero. Let us discuss this point.\par%
We start by some identities. Since $\,(\pa_i \, u_k -\,\pa_k u_i
)\,n_i=\,(\oo \times\,\n)_k\,,\,$ the identity
\begin{equation}\label{frontum}%
(D\,u)\,\cdot\,\n \,\cdot\,v \equiv \,(\pa_k\,u_i+\,\pa_i\,u_k)
n_i\,v_k =\,(\oo \times\,\n)\cdot\,v  +\,2\,(\pa_k u_i ) v_k\,n_i%
\end{equation}
follows. Further, from
\begin{equation}\label{frontdoi}
(\pa_k u_i ) v_k\,n_i =\, \na (u\cdot\,\n)\,\cdot v -\,(\pa_k n_i )
v_k\,u_i\,,
\end{equation}
one gets
\begin{equation}\label{frontre}
(\,D\,u\,)\cdot\,v\cdot\,n =\,(\oo \times\,\n)\cdot\,v\,+\,2\,\na
(u\cdot\,\n)\,\cdot v -\,2\,(\pa_k n_i ) v_k\,u_i\,.
\end{equation}
It follows, in particular, that on flat portions of the boundary the
conditions
\eqref{bcns} and \eqref{bcnos} are equivalent.\par%
By appealing to \eqref{bufe} and \eqref{frontre}, and by assuming
that $\,u\cdot\,n=\,v\cdot\,n=\,0\,$ on $\,\Ga\,,$ we show that
(recall definition \eqref{buh})
\be\label{bufia} \ba{ll}\vspace{1ex}\displaystyle%
\frac12 \,\int_{\O}\ \ B (D\,u)\, \cdot D\,u \cdot D\, v \,dx
=\,-\,\int_{\O}\,\na \cdot\,\big(\,B(D\,u)\,D(u)\,\big)
\cdot\,v \,dx\,  \\
\dy +\,\int_{\Ga} \,B(D\,u)\,(\oo \times\,\n)\cdot\,v\, dS
-\,2\,\int_{\Ga} \,B(D\,u)\,(\pa_k n_i ) v_k\,u_i\, dS\,.%
\ea\ee%
The identity \eqref{bufia} would justify to call $\,u\,$ a weak
solution of problem \eqref{NSC}, \eqref{bcns} if $\,u\,\in\,
V_p(\O)$ satisfies
\be\label{bufies}\ba{ll}\vspace{1ex}\displaystyle%
\frac12 \,\int_{\O}\ \ B (D\,u)\, \cdot D\,u \cdot D\, v \,dx
+\,2\,\int_{\Ga} \,B(D\,u)\,(\pa_k n_i ) v_k\,u_i\, dS= \\
\dy =\,\int_{\O}\, f\, \cdot\, v \,dx\,,
\ea\ee%
for all $\,v \in \,V_p(\O)\,$. Note that if the boundary integral in
equation \eqref{bufies} vanishes for all $\,v\,$ such that
$\,v\cdot\,n=\,0\,$, then $\,B(D\,u)\,(\oo \times\,\n)=\,0\,$ on
$\,\Ga\,.$ Since $\,B(D\,u)\neq \,0\,$, the second boundary
condition \eqref{bcns} follows. However, the boundary integral in
equation \eqref{bufies} is not well defined due to the term
$B(D\,u)\,,$ except for $\,p=\,2\,.$\par%
Note that, in equation \eqref{bufies}, for $\,v=\,u\,$  and  $\O\,$
convex, the integrand in the boundary integral is nonnegative. So,
in a convex domain, we may obtain, at least, an a priori estimate in
$\,W^{1,\,p}(\O)\,$.
\section{On the Fluid Mechanics system}\label{ssetes}
The proof of theorem \ref{teoremaq} may be immediately adapted to
the case $\,q<\,n\,$, also considered in \cite{BVCRI} and
\cite{nonhom}. In the case $\,q<\,n\,$, see \cite{BVCRI}, the
assumption $\,f\in L^q(\O)\,$ does not imply
$\,u\in\,W^{2,q}(\O)\,.$ This last regularity result requires a
stronger assumption on $\,f\,$, namely $\,f\in\,L^{r(q)}(\O)\,,$
where $\,r(q)\,$ is given by
\be\label{rqqq}%
r(q)=\,\frac{nq}{n(p-1)+q(2-p)}\,.%
\ee%
Note that $\,r(q)>\,q\,,$ and $\,r(n)=\,n\,$.\par%
Since, on the whole, regularity results are stronger for large
values of $\,q\,$, in \cite{BVCRI} the authors have assumed, for
convenience, that $\,q\geq\,2\,.$ This assumption excludes the
significant case of square integrable external forces
$\,f\in\,L^2(\O)\,.$ In fact, by \eqref{rqqq}, $\,r(q)=\,2\,$ holds
for $\,q=\,\cq\,$ given by
\begin{equation}\label{errq2}%
\cq =\,\frac{2\,n\,(\,p-\,1\,)}{n-\,2\,(\,2-\-p\,)}<\,2\,.
\end{equation}
However, the proof shown in reference \cite{BVCRI} also applies to
values $\,q<\,2\,$, in particular to $\,\cq\,.$ The (really obvious)
modification required to adapt the proof to this case was shown in
reference \cite{BV-ARX}. In this last reference we were interested
in the particular case $\,r(q)=\,2\,,$ i.e. $\,q=\,\cq\,.$ In
proposition 2.1 in \cite{BV-ARX} we basically remark that if $ f\in
L^2(\O)$ then $u$ belongs to $W^{2,\,\cq}(\O)\,.$ Moreover,
\begin{equation}\label{dnqn}
\|u\|_{2,\cq}\leq C
\,\left(\|f\|_{\cq}+\|f\|_{2}^\frac{1}{p-1}\right)\,.
\end{equation}
On the other hand, during a recent meeting in Levico (December
2012), Jakub Tich\'y informed us about some new results obtained in
collaboration with Petr Kaplick\'y, in reference \cite{kapli}. The
very interesting results obtained by these authors concern the
generalized Stokes problem
\begin{equation}\label{NSCCC}\left\{
\begin{array}{ll}\vspace{1ex}
-\,\nabla \cdot
\,\big(\,(\,\m+|\,D\,u|^2\,)^{\frac{p-2}{2}}\,D\,u\,\big) +\,\na
\,\pi =\,f\,,
\\%
\,\na \cdot\,u=\,0\,,
\end{array}\right.
\end{equation}
under the Navier boundary condition \eqref{bcns}. Actually they
consider more general constitutive relations $\,S(D(u))\,.$ In
particular (see the theorem 1.8 in reference \cite{kapli}), under
natural assumptions on the external force $\,f\,,$ the solution to
problem \eqref{NSCCC}, under the Navier boundary condition,
satisfies
\begin{equation}\label{naosei}
-\,\nabla \cdot
\,\big(\,(\,\m+|\,D\,u|^2\,)^{\frac{p-2}{2}}\,D\,u\,\big)
\in\,\L^2(\O)\,.
\end{equation}
So, by appealing to our theorem \ref{teoremaq} (adapted, as
described above, to the value $\,q=\,\cq\,$), it follows that the
solutions to problem \eqref{NSCCC} belong to $W^{2,\,\cq}(\O)\,.$
Clearly, we have to assume that condition \eqref{cqom} holds for the
value $\,q=\,\cq\,.$

\vspace{0.2cm}

We take the occasion to announce that some new results concerning
the system \eqref{NSCCC} in the torus will be shown in a forthcoming
paper.

\end{document}